\documentclass[12pt]{amsart}

\allowdisplaybreaks

\usepackage{fullpage}

\usepackage{amsfonts,amssymb,amscd,amsmath,enumerate,verbatim,color,xcolor}
\usepackage[latin1]{inputenc}
\usepackage{amscd}
\usepackage{latexsym}

\usepackage{graphicx} 

\usepackage{mathptmx}

\usepackage{hyperref}
\hypersetup{colorlinks,linkcolor=blue ,citecolor=blue, urlcolor=blue}
\def\NZQ{\mathbb}               

\def\RR{{\NZQ R}}

\def\frk{\mathfrak}               

\def\Phi{{\frk N}}
\def\opn#1#2{\def#1{\operatorname{#2}}} 
\opn\gr{gr}

\def\Oc{{\mathcal O}}

\def\Cc{{\mathcal C}}

\newtheorem{Theorem}{Theorem}[section]
\newtheorem{Lemma}[Theorem]{Lemma}
\newtheorem{Corollary}[Theorem]{Corollary}

\theoremstyle{definition}
\newtheorem{Remark}[Theorem]{Remark}

\newtheorem{Definition}[Theorem]{Definition}

\newtheorem{Conjecture}[Theorem]{Conjecture}

\let\epsilon\mu
\let\phi=\varphi
\let\kappa=\varkappa
\opn\dis{dis}
\opn\height{height}
\opn\dist{dist}
\def\pnt{{\raise0.5mm\hbox{\large\bf.}}}

\opn\Lex{Lex}
\opn\conv{conv}

\begin{document}

\title{Triangular faces of the order and chain polytope of\\
a maximal ranked poset}

\author{Aki Mori}

\address{Aki Mori,
	Learning center, Institute for general education, Setsunan University, Neyagawa, Osaka, 572-8508, Japan} 
\email{aki.mori@setsunan.ac.jp}


\begin{abstract}
Let $\Oc(P)$ and $\Cc(P)$ denote the order polytope and chain polytope, respectively, associated with a finite poset $P$. We prove the following result: if $P$ is a maximal ranked poset, then the number of triangular $2$-faces of $\Oc(P)$ is less than or equal to that of $\Cc(P)$, with equality holding if and only if $P$ does not contain an $X$-poset as a subposet.
\end{abstract}

\maketitle
\section{Introduction}
The order polytope $\Oc(P)$ and chain polytope $\Cc(P)$ associated with a finite partially ordered set $P$, were introduced by Stanley \cite{Sta}.
These are significant classes of $n$-dimensional polytopes that have been widely studied in the fields of combinatorics and commutative algebra.
We are interested in the face structure and the number of faces of $\Oc(P)$ and $\Cc(P)$.
For an $n$-dimensional polytope, $(f_0, f_1, \ldots, f_{n-1})$ is called the {\it f-vector} which denotes the number of $i$-dimensional faces, where $0 \leq i \leq n-1$.
Hibi and Li proposed the following conjecture regarding the $f$-vectors of $\Oc(P)$ and $\Cc(P)$:

\begin{Conjecture}[{\cite[Conjecture 2.4]{HL2}}]\label{HL-conj}
Let $P$ be a finite poset with $|P| = d > 1$. 
Then\\
(a)  $f_i(\Oc(P)) \leq f_i(\Cc(P))$ for all $1 \leq i \leq d-1$.\\
(b)  If $f_i(\Oc(P)) = f_i(\Cc(P))$ for some $1 \leq i \leq d-1$, then $\Oc(P)$ and $\Cc(P)$ are unimodularly equivalent.    
\end{Conjecture}

In the same paper \cite{HL2}, it is proven that $\Oc(P)$ and $\Cc(P)$ are unimodularly equivalent if and only if $P$ does not contain an $X$-poset (Figure 1) as a subposet.
As supporting evidence for this conjecture, we present the results that have already been shown.
It is shown that the number of vertices of $\Oc(P)$ is equal to that of $\Cc(P)$ in \cite{Sta}, and similarly, the number of edges of $\Oc(P)$ is equal to that of $\Cc(P)$ in \cite{HLSS}.
In this study, it should be noted that Conjecture \ref{HL-conj} (b) is interpreted as $2 \leq i \leq d-1$ based on the results in \cite{HLSS}.
In \cite{HL2}, it is shown that the number of facets of $\Oc(P)$ is less than or equal to that of $\Cc(P)$, with equality holding if and only if $P$ does not contain an $X$-poset as a subposet.
The recent result described below proves Conjecture \ref{HL-conj} (a) when applied to specific classes of posets.
In \cite{AFJ}, it is shown that if $P$ is a poset called a maximal ranked poset, then $f_{i}(\Oc(P)) \leq f_i(\Cc(P))$ holds for all $1 \leq i \leq d-1$.
Furthermore, in \cite{FL}, it is shown that if $P$ is a poset built inductively by taking disjoint unions and ordinal sums of posets, then $f_{i}(\Oc(P)) \leq f_i(\Cc(P))$ holds for all $1 \leq i \leq d-1$.
The result by \cite{FL} encompasses the result by \cite{AFJ}.
Building upon the previous studies mentioned above, this paper focuses on the triangular $2$-faces of $\Oc(P)$ and $\Cc(P)$.
The main theorem of this paper states that if $P$ is a maximal ranked poset, then the number of triangular $2$-faces of $\Oc(P)$ is less than or equal to that of $\Cc(P)$, with equality holding if and only if $P$ does not contain an $X$-poset as a subposet.
In other words, when $P$ is a maximal ranked poset containing an $X$-poset, the number of triangles in the $1$-skeleton of $\Cc(P)$ is strictly greater than that in the $1$-skeleton of $\Oc(P)$.
This result contributes to the advancement of the Hibi and Li's Conjecture.

\begin{figure}
\centering
\includegraphics[scale=0.5]{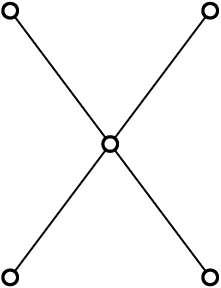}
\caption{$X$-poset}
\end{figure}

\section{Basics on the faces of order and chain polytope}
Let $P=\{p_1,\ldots,p_d\}$ be a finite poset equipped with a partial order $\leq$.
To each subset $W \subset P$, we associate $\rho(W) = \sum_{p_i \in W} {\bf e}_i \in \RR^d$, where ${\bf e}_1, \ldots, {\bf e}_d$ are the canonical unit coordinate vectors of $\RR^d$.
In particular $\rho(\emptyset)$ is the origin of $\RR^d$.
A {\em poset ideal} of $P$ is a subset $I \subset P$ such that if $p_i \in I$ and $p_j \leq p_i$, then $p_j \in I$.
An {\em antichain} of $P$ is a subset $A \subset P$ such that $p_i$ and $p_j$ belonging to $A$ with $i \neq j$ are incomparable.
Note that the empty set $\emptyset$ is considered both a poset ideal and an antichain of $P$.
We say that $p_{j}$ {\em covers} $p_{i}$ if $p_{i} < p_{j}$ and $p_{i}<p_{k}<p_{j}$ for no $p_{k} \in P$. 

The two polytopes associated with a poset, namely the {\em order polytope}
\[
\mathcal{O}(P)
=
\{
(x_{1}, \dots, x_{d}) \in \RR^{d}
:
0 \leq x_{i} \leq 1 
\;
\mbox{for all}
\;
1 \leq i \leq d,
\;
x_{i} \geq x_{j}
\;
\mbox{if}
\;
p_{i} \leq p_{j}
\;
\mbox{in}\;P
\}
\]
and the {\em chain polytope}
\[
\mathcal{C}(P)
=
\{
(x_{1}, \dots, x_{d}) \in \RR^{d}
:
x_{i} \geq 0
\;
\mbox{for all}
\;
1 \leq i \leq d,
\;\;
x_{i_1}+\cdots +x_{i_k} \leq 1
\;
\mbox{if}
\;
p_{i_1} < \cdots < p_{i_k}
\;
\mbox{in} 
\;P
\}
\]
were introduced by \cite{Sta}. 

We briefly review the properties of these two polytopes as established by \cite{Sta}.
One has $\mathrm{dim}{\mathcal O}(P)$ $=$ $\mathrm{dim}{\mathcal C}(P)=d$.
Each vertex of ${\mathcal O}(P)$ corresponds to $\rho(I)$, where $I$ is a poset ideal of $P$ and each vertex of ${\mathcal C}(P)$ corresponds to $\rho(A)$, where $A$ is an antichain of $P$.
Notice that the poset ideals are in a one-to-one correspondence with the antichains.
It then follows that the number of vertices of ${\mathcal O}(P)$ is equal to that of ${\mathcal C}(P)$. 
Then the facets of ${\mathcal O}(P)$ are obtained by the hyperplane whose defining equation is as follows:

\begin{itemize}
\item 
$x_{i}=0$, where $p_{i} \in P$ is maximal;
\item
$x_{j}=1$, where $p_{j} \in P$ is minimal;
\item
$x_{i}=x_{j}$, where $p_{j}$ covers $p_{i}$,
\end{itemize}
\noindent
the facets of ${\mathcal C}(P)$ are the following:

\begin{itemize}
\item 
$x_{i}=0$, for all $p_{i}\in P$;
\item
$x_{i_{1}}+\cdots+x_{i_{k}}=1$, where $p_{i_{1}}<\cdots<p_{i_{k}}$
 is a maximal chain of $P$.
\end{itemize}

Recall that the {\it comparability graph} ${\rm Com}(P)$ of $P$ is the finite simple graph on the vertex set $\{p_1,\ldots, p_d\}$, where the edges are defined by $\{p_i, p_j\}$ with $i \neq j$, such that $p_i$ and $p_j$ are comparable in $P$.
In general, we say that a nonempty subset $Q = \{p_{i_1},\ldots,p_{i_k}\}$ of $P$ is {\it connected} in $P$ if the induced subgraph of ${\rm Com}(P)$ on $\{p_{i_1},\ldots,p_{i_k}\}$ is connected.

The descriptions of edges and triangular $2$-faces of ${\mathcal O}(P)$, as well as those of ${\mathcal C}(P)$, respectively, are obtained as follows.
Note that $A \triangle B$ denotes the symmetric difference of $A$ and $B$, that is $A \triangle B = (A \setminus B)\cup(B \setminus A)$.   

\begin{Lemma}[{\cite[Lemma 4, Lemma 5]{HL1}}]\label{edge}
Let $P$ be a finite poset.
\begin{enumerate}[(a)]
\item
Let $I$ and $J$ be poset ideals of $P$ with $I \neq J$. 
Then the ${\rm conv}(\{\rho(I),\rho(J)\})$ is an edge of ${\mathcal O}(P)$ if and only if $I \subset J$ and $J \setminus I$ is connected in $P$.
\item
 Let $A$ and $B$ be antichains of $P$ with $A \neq B$.
Then the ${\rm conv}(\{\rho(A),\rho(B)\})$ is an edge of ${\mathcal C}(P)$ if and only if $A \triangle B$ is connected in $P$.
\end{enumerate}
\end{Lemma}

\begin{Lemma}[{\cite[Theorem 1, Theorem 2]{M}}]\label{2-face}
Let $P$ be a finite poset.
\begin{enumerate}[(a)]
\item
Let $I$, $J$, and $K$ be pairwise distinct poset ideals of $P$. 
Then the ${\rm conv}(\{\rho(I), \rho(J), \rho(K)\})$ is a 2-face of $\mathcal{O}(P)$ if and only if $I \subset J \subset K$ and $J \setminus I$, $K \setminus J$, $K \setminus I$ are connected in $P$.
\item
 Let $A$, $B$, and $C$ be pairwise distinct antichains of $P$. 
Then the ${\rm conv}(\{\rho(A), \rho(B), \rho(C)\})$ is a 2-face of $\mathcal{C}(P)$ if and only if $A \triangle B$, $B \triangle C$ and $C \triangle A$ are connected in $P$.
\end{enumerate}
\end{Lemma}

\begin{Remark}
The statement of \cite[Theorem 1]{M} is ``$I \subset J \subset K$ and $K \setminus I$ is connected in $P$", which is incorrect. 
Lemma \ref{2-face} implies that the triangles in the $1$-skeleton of $\mathcal{O}(P)$ or $\mathcal{C}(P)$ correspond to the triangular $2$-faces of each polytope.
\end{Remark}

\section{The number of triangular $2$-faces}

The {\em rank} of a poset $P$ is the maximal length of a chain in $P$, where the length of a chain $C$ is $|C|-1$.
A poset $P$ is {\em graded of rank n} if every maximal chain in $P$ has the same length $n$.
Let $P_i$ denote the set of elements of rank $i$.
Namely, if $P$ is graded of rank $n$, then it can be written $P = \cup_{i=0}^n P_i$ such that every maximal chain has the form $p_0 < p_1 < \cdots < p_n$, where $p_i \in P_i$.
If any two elements in distinct rank of a graded poset $P$ are comparable, then $P$ is called a {\em maximal ranked poset} \cite{AFJ}.
Given a poset ideal $I \subset P$, we denote ${\rm max}(I)$ for the set of maximal elements of $I$. Similarly, given an antichain $A \subset P$, we denote $\langle A \rangle$ as the poset ideal of $P$ generated by $A$, that is $\langle A \rangle = \{x : x \leq p\; {\rm for\;some}\; p \in A\}$. This establishes a bijection between poset ideals $I$ and antichains $A$ of $P$.

Let $E^{*}_{\mathcal{O}(P)}$ denote the set of pairs $\{I, J\}$, where $I$ and $J$ are poset ideals of $P$ with $I \neq J$, and ${\rm conv}(\{\rho(I),\rho(J)\})$ forms an edge of $\mathcal{O}(P)$, while ${\rm conv}(\{\rho({\rm max}(I)),\rho({\rm max}(J))\})$ does not form an edge of $\mathcal{C}(P)$.
Let $E^{*}_{\mathcal{C}(P)}$ denote the set of pairs $\{A, B\}$, where $A$ and $B$ are antichains of $P$ with $A \neq B$, and ${\rm conv}(\{\rho(A),\rho(B)\})$ forms an edge of $\mathcal{C}(P)$, while ${\rm conv}(\{\rho(\langle A \rangle),\rho(\langle B \rangle)\})$ does not form an edge of $\mathcal{O}(P)$.

\begin{Lemma}\label{OCedge}
Let $P$ be a maximal ranked poset of rank $n$.
\begin{enumerate}[(a)]
\item
Let $I$ and $J$ be poset ideals of $P$ with $I \neq J$. 
Then $\{I, J\} \in E^{*}_{\mathcal{O}(P)}$ if and only if $I = \emptyset$, $J$ is connected in $P$, and $|{\rm max}(J)| \geq 2$.  
\item
Let $A$ and $B$ be antichains of $P$ with $A \neq B$.
Then $\{A, B\} \in E^{*}_{\mathcal{C}(P)}$ if and only if there exists 
 some $\ell$ such that $A = P_{\ell -1}$, $B \subset P_{\ell}$, and $|B| \geq 2$, where $1 \leq \ell \leq n$.  
\end{enumerate}
\end{Lemma}

\begin{proof}
(a)
{\bf (``If")}
From Lemma \ref{edge}(a), if $I = \emptyset$ and $J$ is connected, then ${\rm conv}(\{\rho(I),\rho(J)\})$ forms an edge of $\mathcal{O}(P)$.
Since $|{\rm max}(J)| \geq 2$, ${\rm max}(I)\,\triangle\,{\rm max}(J)={\rm max}(J)$ is disconnected.
Thus, by Lemma \ref{edge}(b), ${\rm conv}(\{\rho({\rm max}(I)),\rho({\rm max}(J))\})$ does not form an edge of $\mathcal{C}(P)$.
\\
{\bf (``Only if")}
According to Lemma \ref{edge}, $\{I, J\} \in E^{*}_{\mathcal{O}(P)}$ if and only if $I \subset J$, $J \setminus I$ is connected, and ${\rm max}(I)\,\triangle\,{\rm max}(J)$ is disconnected in $P$.
If ${\rm max}(I)\triangle{\rm max}(J)$ is disconnected, then  either ${\rm max}(I)=\emptyset$ and $|{\rm max}(J)| \geq 2$, or ${\rm max}(I),{\rm max}(J)$ are subsets of $P_k$, where $0 \leq k \leq n$, since $P$ is a maximal ranked poset.
If ${\rm max}(I),{\rm max}(J) \subset P_k$, then 
$
{\rm max}(I)\;\triangle\;{\rm max}(J)
=
{\rm max}(J)\setminus{\rm max}(I)
=
J \setminus I$,
since $I \subset J$, namely ${\rm max}(I)\subset{\rm max}(J)$.
Since $J \setminus I$ is connected and ${\rm max}(I)\;\triangle\,{\rm max}(J)$ is disconnected, this leads to a contradiction.
Hence, we have ${\rm max}(I)=\emptyset$ and $|{\rm max}(J)| \geq 2$, as desired. 

\medskip
(b)
{\bf (``If")}
If there exists some $\ell$ such that $A = P_{\ell -1}$, $B \subset P_{\ell}$, and $|B| \geq 2$, where $1 \leq \ell \leq n$, then $A \triangle B$ is connected.
Then, from Lemma \ref{edge}(b), ${\rm conv}(\{\rho(A),\rho(B)\})$ forms an edge of $\mathcal{C}(P)$.
Furthermore, $\langle A \rangle \subset \langle B \rangle$ and $\langle B \rangle \setminus \langle A \rangle$ is disconnected, because $\langle B \rangle \setminus \langle A \rangle = B$ and $|B| \geq 2$.
Hence, by Lemma \ref{edge}(a), ${\rm conv}(\{\rho(\langle A \rangle),\rho(\langle B \rangle))$ does not form an edge of $\mathcal{O}(P)$.
\\
{\bf (``Only if")}
According to Lemma \ref{edge}, $\{A, B\} \in E^{*}_{\mathcal{C}(P)}$ if and only if the following conditions hold:
\begin{itemize}
\item
$A \triangle B$ is connected in $P$ ;
\item
$\langle A \rangle \setminus \langle B \rangle \neq \emptyset$ and $\langle B \rangle \setminus \langle A \rangle \neq \emptyset$ or $\langle A \rangle \subset \langle B \rangle$ and $\langle B \rangle \setminus \langle A \rangle$ is disconnected in $P$.
\end{itemize}
If $\langle A \rangle \setminus \langle B \rangle \neq \emptyset$ and $\langle B \rangle \setminus \langle A \rangle \neq \emptyset$, then there exists some $k$ such that $A,B \subset P_k$ and $A \setminus B \neq \emptyset$ and $B \setminus A \neq \emptyset$, where $0 \leq k \leq n$, because $P$ is a maximal ranked poset.
This contradicts $A \triangle B$ is connected.
Thus, assuming $\langle A \rangle \subset \langle B \rangle$ and $\langle B \rangle \setminus \langle A \rangle$ is disconnected, there exists some $\ell$ such that $\langle B \rangle \setminus \langle A \rangle \subset P_{\ell}$ and $|\langle B \rangle \setminus \langle A \rangle| \geq 2$.
In this case, either $A, B \subset P_{\ell}$ or $A=P_{\ell -1},B\subset P_{\ell}$ holds, but $A, B \subset P_{\ell}$ contradicts $A \triangle B$ is connected.
Hence, there exists some $\ell$ such that $A=P_{\ell -1},B\subset P_{\ell}$ and $|\langle B \rangle \setminus \langle A \rangle| = |B| \geq 2$, where $1 \leq  \ell \leq n$. 
\end{proof}

\begin{Definition}
Let $P$ be a finite poset.
Let $I$, $J$, and $K$ be pairwise distinct poset ideals of $P$, and
let $A$, $B$, and $C$ be pairwise distinct antichains of $P$.
We define the sets of triples corresponding to triangular $2$-faces of $\mathcal{O}(P)$ and $\mathcal{C}(P)$ as follows:
\begin{eqnarray*}
\Delta_{{\mathcal{O}}(P)}
&:=&
\bigl\{\{I,J,K\}\;
:
\;{\rm conv}(\{\rho(I), \rho(J), \rho(K)\})\; \mbox{is a 2-face of}\;\mathcal{O}(P)\bigr\};
\\
\Delta_{{\mathcal{C}}(P)}
&:=&
\bigl\{\{A,B,C\}\;
:
\;{\rm conv}(\{\rho(A), \rho(B), \rho(C)\})\;\mbox{is a 2-face of}\;\mathcal{C}(P)\bigr\};
\\
\Delta^{*}_{{\mathcal{O}}(P)}
&:=&
\bigl\{\{I,J,K\} \in \; \Delta_{\mathcal{O}(P)}
:
\mbox{at least one of}\;\{I,J\}, \{J,K\},\;\mbox{or}\; \{I,K\}\;\mbox{belongs to}\; E^{*}_{\mathcal{O}(P)}\bigr\};
\\
\Delta^{*}_{{\mathcal{C}}(P)}
&:=&
\bigl\{\{A,B,C\} \in \; \Delta_{\mathcal{C}(P)}
:
\mbox{at least one of}\;\{A,B\}, \{B,C\},\;\mbox{or}\; \{A,C\}\;\mbox{belongs to}\; E^{*}_{\mathcal{C}(P)}\bigr\}.
\end{eqnarray*}
\end{Definition}

\medskip
The first claim we wish to present is that the number of triangular $2$-faces of $\Oc(P)$ is less than or equal to that of $\Cc(P)$, namely $|\Delta_{{\mathcal{O}}(P)}| \leq |\Delta_{{\mathcal{C}}(P)}|$.
This assertion is guided by the following two lemmas.

\begin{Lemma}\label{triangle equality}
Let $P$ be a finite poset.
Then the number of triples in $\Delta_{\mathcal{O}(P)}\setminus\Delta^{*}_{\mathcal{O}(P)}$
equals that in $\Delta_{\mathcal{C}(P)}\setminus\Delta^{*}_{\mathcal{C}(P)}$.
\end{Lemma}

\begin{proof}
From Lemma \ref{2-face}, the set of triples $\{\rho(I),\rho(J),\rho(K)\}$ where $\{I,J,K\} \in \Delta_{\mathcal{O}(P)}\setminus\Delta^{*}_{\mathcal{O}(P)}$, corresponds to the set of triangles in the graph obtained by removing the all edges $\{\rho(I_1),\rho(I_2)\}$, where $\{I_1,I_2\} \in E^{*}_{\mathcal{O}(P)}$ from the $1$-skeleton of ${\mathcal{O}}(P)$.
Similarly, the set of triples $\{\rho(A),\rho(B),\rho(C)\}$ where $\{A,B,C\} \in \Delta_{\mathcal{C}(P)}\setminus\Delta^{*}_{\mathcal{C}(P)}$, corresponds to the set of triangles in the graph obtained by removing the all edges $\{\rho(A_1),\rho(A_2)\}$, where $\{A_1,A_2\} \in E^{*}_{\mathcal{C}(P)}$ from the $1$-skeleton of ${\mathcal{C}}(P)$.
Since these graphs are clearly isomorphic, the number of triangles in them coincides.
\end{proof}

\begin{Lemma}\label{triangle inequality} 
Let $P$ be a maximal ranked poset.
Then the number of triples in $\Delta^{*}_{\mathcal{O}(P)}$
is less than or equal to that in $\Delta^{*}_{\mathcal{C}(P)}$.
\end{Lemma}

\begin{proof}
From Lemma \ref{2-face}(a) and Lemma \ref{OCedge}(a), the following conditions hold for $\{I,J,K\} \in \Delta^{*}_{\mathcal{O}(P)}$, where $I$, $J$, and $K$ be pairwise distinct poset ideals of $P$
:
\begin{itemize}
\item 
$I = \emptyset$, $J \subset K$ and $J$, $K$, and $K \setminus J$ are connected in $P$;
\item 
$|{\rm max}(J)| \geq 2$ or $|{\rm max}(K)| \geq 2$.
\end{itemize}
Let $P$ has rank $n$.
Since $P$ is a maximal ranked poset, we may assume that ${\rm max}(J) \subset P_j$ and ${\rm max}(K) \subset P_k$, where $0 \leq j \leq k \leq n$.
We define the map $\phi\;:\;\Delta^{*}_{\mathcal{O}(P)} \rightarrow \Delta^{*}_{\mathcal{C}(P)}$ by
\[
\{I,J,K\} \mapsto 
\begin{cases}
\{P_{j-1},\; {\rm max}(J),\; {\rm max}(K)\} & \mbox{if}\;\; |{\rm max}(J)| \geq 2,\\
\{P_{k-1},\; {\rm max}(J),\; {\rm max}(K)\} & \mbox{if}\;\; |{\rm max}(J)| = 1 \;\mbox{and}\; k-j \neq 1,\\
\{P_{k-1},\; {\rm min}(K \setminus J),\; {\rm max}(K)\} & \mbox{if}\;\; |{\rm max}(J)| = 1 \;\mbox{and}\; k-j = 1.
\end{cases}
\]
We will show that the map is an injection.
\\
\\
\medskip
\noindent
{\bf Case 1.} $|{\rm max}(J)| \geq 2.$\\
\indent
Since $J$ is connected, we may assume that $j \geq 1$.
Now, $P_{j-1}$, ${\rm max}(J)$ and ${\rm max}(K)$ are pairwise distinct antichains of $P$. 
As $P$ is a maximal ranked poset, $P_{j-1} \triangle {\rm max}(J)$ and $P_{j-1} \triangle {\rm max}(K)$ are connected.
If $j < k$, then ${\rm max}(J) \triangle {\rm max}(K)$ is connected.
If $j = k$, then ${\rm max}(J)\;\triangle\;{\rm max}(K)={\rm max }(K)\setminus{\rm max}(J)=K \setminus J,$ because $J \subset K$, namely ${\rm max}(J)\subset{\rm max}(K)$.
Since $K \setminus J$ is connected, ${\rm max}(J) \triangle {\rm max}(K)$ is  also connected. 
Furthermore, since $|{\rm max}(J)| \geq 2$, from Lemma \ref{OCedge}(b), we have $\{P_{j-1}, {\rm max}(J)\} \in E^{*}_{\mathcal{C}(P)}$, implying $\{P_{j-1},{\rm max}(J),{\rm max}(K)\} \in \Delta^{*}_{\mathcal{C}(P)}$.

\medskip
\noindent
{\bf Case 2.} $|{\rm max}(J)| = 1$ and $k-j \neq 1$.\\
\indent
Now, $|{\rm max}(K)| \geq 2$.
Since $P$ is a maximal ranked poset and $k - j \neq 1$, ${\rm max}(J) \not\subset P_{k-1}$.
Thus $P_{k-1}$, ${\rm max}(J)$ and ${\rm max}(K)$ are pairwise distinct antichains of $P$ and $P_{k-1} \triangle {\rm max}(J)$ and $P_{k-1} \triangle {\rm max}(K)$ are connected.
If $k-j \geq 2$, then ${\rm max}(J) \triangle {\rm max}(K)$ are connected.
If $k-j = 0$, then we can show ${\rm max}(J) \triangle {\rm max}(K)$ is connected by the same argument in Case 1.
Furthermore, since $|{\rm max}(K)| \geq 2$, from Lemma \ref{OCedge}(b), we have $\{P_{k-1}, {\rm max}(K)\} \in E^{*}_{\mathcal{C}(P)}$, implying $\{P_{k-1},{\rm max}(J),{\rm max}(K)\} \in \Delta^{*}_{\mathcal{C}(P)}$.

\medskip
\noindent
{\bf Case 3.} $|{\rm max}(J)| = 1$ and $k-j = 1$.\\
\indent
Now, $|{\rm max}(K)| \geq 2$.
Since $P$ is a maximal ranked poset and $k - j = 1$, ${\rm min}(K \setminus J) = P_j \setminus {\rm max}(J) = P_{k-1} \setminus {\rm max}(J) \subset P_{k-1}$.
Thus $P_{k-1}$, ${\rm min}(K \setminus J)$ and ${\rm max}(K)$ are pairwise distinct antichains of $P$ and $P_{k-1} \triangle {\rm max}(K)$ and ${\rm min}(K \setminus J)\triangle {\rm max}(K)$ are connected.
Since ${\rm min}(K \setminus J) = P_{k-1} \setminus {\rm max}(J)$, then $P_{k-1} \triangle {\rm min}(K \setminus J) = {\rm max}(J)$.
Hence, $P_{k-1} \triangle {\rm min}(K \setminus J)$ is connected because $|{\rm max}(J)|=1$.
Furthermore, since $|{\rm max}(K)| \geq 2$, from Lemma \ref{OCedge}(b), we have $\{P_{k-1}, {\rm max}(K)\} \in E^{*}_{\mathcal{C}(P)}$, implying $\{P_{k-1},{\rm max}(J),{\rm max}(K)\} \in \Delta^{*}_{\mathcal{C}(P)}$. 
\\

\medskip
In Case 1 or Case 2, the map is clearly injective.
In Case 3, let $\{\emptyset, J, K \}$ and $\{\emptyset, J', K' \}$ belong to $\Delta^{*}_{\mathcal{O}(P)}$.
We note the following holds:
\begin{itemize}
\item 
${\rm max}(J') \subset P_{j'}$ and ${\rm max}(K') \subset P_{k'}$, where $0 \leq j' \leq k' \leq n$;
\item
$|{\rm max}(J')|=1$ and $k'-j'=1$.
\end{itemize}
Suppose that $P_{k-1} = P_{k'-1}$, ${\rm min}(K \setminus J)={\rm min}(K' \setminus J')$, and ${\rm max}(K)={\rm max}(K')$.
Then $K=K'$.
Since $k - j = k' - j' = 1$, ${\rm max}(K \setminus J) = {\rm max}(K) = {\rm max}(K') = {\rm max}(K' \setminus J')$.
Thus ${\rm max}(K \setminus J) = {\rm max}(K' \setminus J')$.
Now, we have ${\rm min}(K \setminus J) = {\rm min}(K \setminus J')$ and ${\rm max}(K \setminus J) = {\rm max}(K \setminus J')$ by $K=K'$.
Since $K \setminus J = {\rm min}(K \setminus J) \cup {\rm max}(K \setminus J)$, it follows that $K \setminus J = K \setminus J'$, namely $J=J'$, as desired.
\end{proof}

\medskip
From Lemma \ref{triangle equality} and Lemma \ref{triangle inequality}, it follows that $|\Delta_{\mathcal{O}(P)}| \leq |\Delta_{\mathcal{C}(P)}|$.
Consequently, we obtain the following.

\begin{Corollary}
Let $P$ be a maximal ranked poset.
Then the number of triangular $2$-faces of $\mathcal{O}(P)$ is less than or equal to that of $\mathcal{C}(P)$.
\end{Corollary}

Furthermore, we can characterize whether the number of triangular $2$-faces of $\Oc(P)$ is equal to that of $\Cc(P)$, namely $|\Delta_{\mathcal{O}(P)}| = |\Delta_{\mathcal{C}(P)}|$, by the $X$-poset from Figure 1.

\begin{Theorem}\label{main}
Let $P$ be a maximal ranked poset.
Then the number of triangular $2$-faces of $\mathcal{O}(P)$ is equal to that of $\mathcal{C}(P)$ if and only if the poset $P$ does not contain an $X$-poset as a subposet.
\end{Theorem}

\begin{proof}
{\bf (``If")}
The poset $P$ does not contain an $X$-poset from as a subposet if and only if $\mathcal{O}(P)$ and $\mathcal{C}(P)$ are unimodularly equivalent by \cite[Theorem 2.1]{HL2}.
Thus, the $1$-skeleton of $\mathcal{O}(P)$ and that of $\mathcal{C}(P)$ are isomorphic as finite graphs.
In particular, the number of triangles in those $1$-skeletons coincides. 
Therefore $|\Delta_{\mathcal{O}(P)}| = |\Delta_{\mathcal{C}(P)}|$.
\\
{\bf (``Only if")}
Suppose $P$ has rank $n$, and that the poset $P$ contains an $X$-poset as a subposet.
Let the $X$-poset be $\{a, b, c, d, e\}$, where each element is pairwise distinct such that (i) $a$ and $b$ are incomparable, (ii) $d$ and $e$ are incomparable, and (iii) $a < c < d$, $b < c < e$.
We may assume that $d, e \in P_{s}$, $ c \in P_{s - 1}$ and $a, b \in P_{s - t}$, where $2 \leq s \leq n$ and $2 \leq t \leq s$.
Now, $P_{s - t}$, $P_{s -1}$, and $P_{s}$ are pairwise distinct antichains.
As $P$ is a maximal ranked poset, $P_{s - t} \triangle P_{s -1}$, $P_{s - 1} \triangle P_{s}$, and $P_{s - t} \triangle P_{s}$ are connected in $P$.
Since $|P_{s}| \geq 2$, from Lemma \ref{OCedge}(b), $\{P_{s -1}, P_{s}\} \in E^{*}_{\mathcal{C}(P)}$, implying $\{P_{s- t}, P_{s - 1}, P_{s}\} \in \Delta^{*}_{\mathcal{C}(P)}$. 
We will show that there does not exist any $\{I, J, K\} \in \Delta^{*}_{\mathcal{O}(P)}$ such that $\phi(\{I, J, K\}) = \{P_{s - t}, P_{s - 1}, P_{s}\}$, where $\phi$ is the injection obtained in the proof of Lemma \ref{triangle inequality}.

\medskip
\noindent
{\bf Case 1.}
Suppose that there exists some $\{I, J, K\} \in \Delta^{*}_{\mathcal{O}(P)}$ such that $|{\rm max}(J)| \geq 2$ and $\phi(\{I, J, K\})$ coincides $\{P_{s - t}, P_{s - 1}, P_{s}\}$.
Since $|{\rm max}(J)| \geq 2$, then $\phi(\{I, J, K\}) = \{P_{j-1},\; {\rm max}(J),\; {\rm max}(K)\}$.
We may assume that $P_s = {\rm max}(K)$, $P_{s-1} = {\rm max}(J)$, and $P_{s-t} = P_{j-1}$.
Then $K = \langle P_{s} \rangle$ and $J = \langle P_{s-1} \rangle$.
However, since $\langle P_{s} \rangle \setminus \langle P_{s -1} \rangle$ is disconnected, this contradicts $K \setminus J$ is connected, because $\{I, J, K\} \in \Delta^{*}_{\mathcal{O}(P)}$.

\medskip
\noindent
{\bf Case 2.}
Suppose that there exists some $\{I, J, K\} \in \Delta^{*}_{\mathcal{O}(P)}$ such that $|{\rm max}(J)| = 1$, $k - j \neq 1$, and $\phi(\{I, J, K\})$ coincides $\{P_{s - t}, P_{s - 1}, P_{s}\}$.
Since $|{\rm max}(J)| = 1$ and $k - j \neq 1$, then $\phi(\{I, J, K\}) = \{P_{k-1},\; {\rm max}(J),\; {\rm max}(K)\}$.
We may assume that $P_s = {\rm max}(K)$, $P_{s-1} = P_{k-1}$, and $P_{s-t} = {\rm max}(J)$.
However, Since $|P_{s-t}| \geq 2$, this contradicts $|{\rm max}(J)| = 1$.

\medskip
\noindent
{\bf Case 3.}
Suppose that there exists some $\{I, J, K\} \in \Delta^{*}_{\mathcal{O}(P)}$ such that $|{\rm max}(J)| = 1$, $k - j = 1$, and $\phi(\{I, J, K\})$ coincides $\{P_{s - t}, P_{s - 1}, P_{s}\}$.
Since $|{\rm max}(J)| = 1$ and $k - j = 1$, then $\phi(\{I, J, K\}) = \{P_{k-1},\; {\rm min}(K \setminus J),\; {\rm max}(K)\}$.
We may assume that $P_s = {\rm max}(K)$, $P_{s-1} = P_{k-1}$, and $P_{s-t} = {\rm min}(K \setminus J)$.
Then $K = \langle P_{s} \rangle$ and $J = \langle P_{s-t-1} \rangle$.
However, Since $s - (s-t-1) = t+1 \geq 3$, this contradicts $k - j = 1$. 

\medskip
By the above arguments, the map $\phi$ turns out to be not surjective.
Consequently, it follows that $|\Delta^{*}_{\mathcal{O}(P)}| < |\Delta^{*}_{\mathcal{C}(P)}|$.
From Lemma \ref{triangle equality}, we establish $|\Delta_{\mathcal{O}(P)}| < |\Delta_{\mathcal{C}(P)}|$, as desired.
\end{proof}

In the following statement, we set the $X$-poset to be $\{a, b, c, d, e\}$, where $d, e \in P_{s}$, $ c \in P_{s - 1}$, and $a, b \in P_{s - t}$, with $2 \leq s \leq n$ and $2 \leq t \leq s$, following the same notation as in the proof of Theorem \ref{main}. 

\begin{Corollary}
Let $P$ be a maximal ranked poset of rank $n$ which contains an $X$-poset as a subposet.
The number of triangular $2$-faces of $\mathcal{C}(P)$ is equal to that of $\mathcal{O}(P)$ plus
\[\sum_{s=2}^n \; \sum_{t=2}^s \; \sum_{m=2}^{|P_s|} \;\sum_{\ell = 2}^{|P_{s-t}|} \binom{|P_s|}{m} \binom{|P_{s-t}|}{\ell}.\]
\end{Corollary}

\begin{proof}
Let $\{d, e\} \subset P'_{s} \subset P_{s}$ and $\{a, b\} \subset P'_{s-t} \subset P_{s-t}$.
Then $\{P'_{s-t}, P_{s-1}, P'_{s}\} \in \Delta^{*}_{\mathcal{C}(P)}$, and there does not exist any $\{I, J, K\} \in \Delta^{*}_{\mathcal{O}(P)}$ such that $\phi(\{I, J, K\}) = \{P'_{s - t}, P_{s - 1}, P'_{s}\}$, by the same argument as in the  proof of Theorem \ref{main}.
Fixing $s$, there are $\sum_{m =2}^{|P_s|} \binom{|P_s|}{m}$ ways to choose $P'_{s}$ and $\sum_{t=2}^{s} \sum_{\ell = 2}^{|P_{s-t}|} \binom{|P_{s-t}|}{\ell}$ ways to choose $P'_{s-t}$.
Therefore,
\[\sum_{s=2}^n\biggl\{\sum_{m =2}^{|P_s|} \binom{|P_s|}{m}\;\sum_{t=2}^{s} \sum_{\ell = 2}^{|P_{s-t}|} \binom{|P_{s-t}|}{\ell}\biggr\}
=
\sum_{s=2}^n \; \sum_{t=2}^s \; \sum_{m=2}^{|P_s|} \;\sum_{\ell = 2}^{|P_{s-t}|} \binom{|P_s|}{m} \binom{|P_{s-t}|}{\ell}.\]

\end{proof}

\section*{Acknowledgement}
The author would like to thank Takayuki Hibi for numerous helpful discussions.

\end{document}